\documentclass[12pt]{article}

\usepackage{amsfonts}
\usepackage{amsxtra}
\usepackage{amsthm}
\usepackage{amssymb}
\usepackage{amsmath,amscd}
\usepackage{supertabular}
\usepackage{pstricks}
\usepackage{pst-plot}
\usepackage{pstricks-add}
\usepackage{pst-eps}
\usepackage{makeidx}
\usepackage{bm}
\usepackage{color} 
\usepackage{hyper ref} 

\newtheorem{theorem}{Theorem}[section]

\newtheorem{definition}[theorem]{Definition}
\newtheorem{proposition}[theorem]{Proposition}

\newtheorem{question}[theorem]{Question}

\newtheorem{example}[theorem]{Example}

\DeclareMathOperator\tr{tr}

\makeindex

\begin{document}

\title{{Angular Momenta of Relative Equilibrium Motions and Real Moment Map Geometry}}

\author
{Gert Heckman and Lei Zhao}

\date{\today}

\maketitle

\begin{abstract}
In  \cite{Chenciner--Jiminez Perez 2011}, Chenciner and Jim\'enez-P\'erez showed 
that the range of the spectra of the angular momenta of all the rigid motions 
of a fixed central configuration in a general Euclidean space form 
a convex polytope. In this note we explain how this result follows from a 
general ``real'' convexity theorem of O'Shea and Sjamaar in symplectic geometry
\cite{O'Shea--Sjamaar 2000}. Finally, we provide a representation-theoretic description of the 
pushforward of the normalized measure under the real moment map for Riemannian symmetric pairs.
\end{abstract}

\section{Introduction}

An $n$-body configuration $x=(x_1,\cdots,x_n)$ in a Euclidean space $E$ with
masses $m_1,\cdots,m_n>0$ moving in a Newtonian force field $F=\nabla U (x)$
with reduced center of mass $\sum m_kx_k/\sum m_k=0$ is called 
\emph{balanced} with factor $\Lambda$ if
\[ \nabla U(x)=-\Lambda xm \]
for $\Lambda:E\rightarrow E$ a symmetric linear operator on $E$ and 
$m=\mathrm{diag}(m_1,\cdots,m_n)$ the mass matrix 
\cite{Albouy--Chenciner 1998}. This is an algebraic equation with presumably 
an abundance of solutions for large $n$.
It is clear that for $\mu>0$ and $k\in\mathrm{SO}(E)$ the similar configuration 
$\mu kx$ is again balanced with factor $\Lambda/\mu^3$.

Let $x$ be such a balanced configuration in $E$ with factor $\Lambda$.
If $Z:E\rightarrow E$ is a skew symmetric linear operator and satisfies $Z^2=-\Lambda$, then the rigid motion 
$t\mapsto z(t)=\exp(tZ)x$ is a solution of Newton's equation 
$$\ddot{z}m=\nabla U(z).$$
Chenciner and Jim\'enez-P\'erez have shown that
the range of the  spectra of the angular momenta of all such rigid 
motions is a convex polytope \cite{Chenciner--Jiminez Perez 2011}, {which is subsequently 
used by Chenciner in the analysis of bifurcation of relative equilibrium motions of the n-body problem \cite{Chenciner 2014}. }

In this note, we will show that this result is just an immediate consequence
of a convexity theorem of O'Shea and Sjamaar in real moment map 
geometry \cite{O'Shea--Sjamaar 2000}, which will be reviewed in particular in the setting of a pair of real reductive Lie algebras. We have made an effort to to write a
pedestrian exposition. For this reason, we have 
restricted ourselves to the case of \emph{central} configurations, 
for which $\Lambda=\lambda$ is just a scalar operator. 
Indeed, the analysis of the spectra range of the angular momentum of a balanced configuration
breaks down to this case, as has been explained in \cite{Chenciner 2010}.

Finally, one may naturally ask about the density of complex structures corresponding to the same spectrum in the range of the spectra of the angular momenta of rigid motions. An explicit description of this density requires more involved work and will be a question for future research. Nevertheless, motivated by this, and in line with O'Shea-Sjamaar's study of the real moment map, we shall give a description of the pushforward of the normalized invariant measure by the real moment map for Riemannian symmetric pairs. 
This provides a real version of the push-forward of the Liouville measure by the moment map \cite{Duistermaat--Heckman 1982} in this special case. 

\section{The $n$-body problem in Euclidean space of arbitrary finite dimension}

The Newtonian $n$-body problem in a finite dimensional Euclidean space $E$ with inner product
$(\cdot,\cdot)$ is the study of the dynamics of $n$ point particles with positions $x_k\in E$ 
and masses $m_k>0$, with time evolution according to Newton's laws of motion
\[ m_k\ddot{x}_k=\sum_{j\neq k}m_jm_k(x_j-x_k)/|x_j-x_k|^3 \]
for $k=1,\cdots,n$. 
A configuration $x=(x_1,\cdots,x_n)$ in $E^n$ is a row vector with entries vectors in $E$. 
Its dual configuration $x^{\ast}$ on $E^n$ is then a column vector with the corresponding dual vectors on $E$ as entries.
Here any vector in $E$ gives rise to a dual vector on $E$ by taking the inner product with that vector.
For example, with this notation $x^{\ast}x$ is the $n\times n$ Gram matrix of the position configuration,
while $xmx^{\ast}$ is the symmetric linear operator on $E$ sending $v$ to $\sum_km_k(x_k,v)x_k$.
Here $m=\mathrm{diag}(m_1,\cdots,m_n)$ is the mass matrix. 

The negative of the potential energy (which is also called the \emph{force function} by Lagrange)
\[ U(x)=\sum_{j<k}m_jm_k/|x_j-x_k| \]
is a solution of the equations
\[ \nabla_kU(x)=\sum_{j\neq k}m_jm_k(x_j-x_k)/|x_j-x_k|^3 \]
with $\nabla_k$ the gradient with respect to the vector $x_k\in E$.
If we denote $\nabla U(x)=(\nabla_1U(x),\cdots,\nabla_nU(x))\in E^n$, then the equations
of motion can be written in the form
\[ \dot{x}=y,\dot{y}m=\nabla U(x) \]
as a first order system. We denote $K(y)=\tr(y^{\ast}ym)/2$ for the kinetic energy. 
The total energy is thus defined by $H(x,y)=K(y)-U(x)$, and is a conserved quantity:
Indeed, we have $\dot{H}=\tr(y^{\ast}\dot{y}m)-\tr(\dot{x}^{\ast}\nabla U(x))=0$.

The total linear momentum $p=\sum m_ky_k\in E$ is also conserved, which in turn implies that the center of mass
$c=\sum m_kx_k/\sum m_k\in E$ has uniform rectilinear motion. By the center of mass reduction
we may assume that $c=p=0$, \emph{which will be done from now on}.

For the position-velocity pair $(x,y)\in E^n\times E^n$, the total angular momentum is defined by
\[ L=ymx^{\ast}-xmy^{\ast}, \]
which is a skew symmetric linear operator on $E$.
Since 
$$\dot{L}=\dot{y}mx^{\ast}-xm\dot{y}^{\ast}=(\nabla U(x))x^{\ast}-x(\nabla U(x))^{\ast}$$
is the linear operator on $E$ sending $v$ to
\[ \sum_{j\neq k}m_jm_k(x_k,v)\frac{x_j-x_k}{|x_j-x_k|^3}-\sum_{j\neq k}m_jm_k\frac{(x_j,v)-(x_k,v)}{|x_j-x_k|^3}x_k=0 \]
we conclude that $L$ is conserved. The conservation of total linear momentum and of total angular momentum
is a consequence of the Euclidean motion group of $E$ being symmetry group of the equations of motion, in accordance
with the Noether theorem. 

For $n\geq3$, the system is non-integrable in the sense that there are no other independent integrals of motion than the above, a result for 
algebraic integrals due to Bruns in 1887 \cite{Bruns 1887} (substantially completed and generalized in \cite{Julliard Tosel 2000}) and for analytic integrals due to Poincar\'{e} 
in 1890 \cite{Siegel--Moser 1971}. This work by Poincar\'{e} on the (restricted)
3-body problem
reveals the great complexity of the general motion in case $n\geq3$ \cite{Poincare MNMC}.

For $n=2$ the relative position $z=x_1-x_2 \in E$ is a solution of the Kepler problem 
\[ \mu\ddot{z}=-\kappa z/|z|^3\;\;\Leftrightarrow\;\;\ddot{z}=-\lambda z/|z|^3 \]
with $\kappa=m_1m_2,\lambda=m_1+m_2,\mu=\kappa/\lambda>0$. For $H=\mu|\dot{z}|^2/2-\kappa/|z|<0$ the motion is
bounded inside the region $|z|\leq-\kappa/H$, and is either collinear or the point $z$ moves in the Euclidean 
plane $P$ spanned by $z$ and $v=\dot{z}$ along an ellipse with a focus at the origin, according to the area law.
Let $i$ be a complex structure on $P$ compatible with the Euclidean structure, which means that 
$i:P\rightarrow P$ is a skew symmetric linear operator with $i^2=-1$. In polar coordinates $(r, \theta)$, the complex variable
\[ z=re^{i\theta} \]
is a solution of the Kepler problem if and only if $(r,\theta)$ is a solution of
\[ \ddot{r}-r\dot{\theta}^2=-\lambda/r^2\;,\;r\ddot{\theta}+2\dot{r}\dot{\theta}=0\;. \]
For $\dot{\theta}=0$ we get the one dimensional Kepler problem $\ddot{r}=-\lambda/r^2$, 
which corresponds to homothetic motion of $z$ in $E$. For $\dot{r}=0$ we find
$\dot{\theta}^2=\lambda/r^3$, which corresponds to rigid uniform circular motion with 
angular velocity $\omega=\sqrt{\lambda/r^3}$.  

For special initial configurations $x\in E^n$, there exists initial configurations of velocities $y \in E^{n}$ such that the above-mentioned Kepler orbits can be lifted to
exact solutions of the $n$-body problem in $E$. These are the so called central configurations 
and give rise to homographic motions. They generalize the collinear $3$-body configurations
of Euler from 1767 \cite{Euler 1767} and the planar equitriangular $3$-body configurations of Lagrange 
from 1772 \cite{Lagrange 1772}. Later examples were found for $n=4$ by Lehmann-Filk\'{e}s in 1891 
\cite{Lehmann 1891}, and Moulton in \cite{Moulton 1910}, and the abundance of planar central configurations for large $n$ was indicated 
by Dziobek, who also started to use the term ``central figure for such a configuration" in 1899 \cite{Dziobek 1899}. 

Planar and spatial central configurations became a renown subject in celestial mechanics, 
notably after the standard text book of Wintner from 1941 \cite{Wintner 1941} and a crucial paper 
by Smale from 1970 \cite{Smale 1970}. The question of linear stability for some planar central 
configurations was undertaken by Moeckel in the eighties and nineties, generalizing the Gascheau 
stability condition from 1843 for the Lagrange equilateral triangle configuration 
\cite{Gascheau 1843},\cite{Routh 1875},\cite{Moeckel 1994},\cite{Moeckel 1997}.
Central configurations in a Euclidean space $E$ of arbitrary finite dimension 
were considered by Albouy and Chenciner in 1998 \cite{Albouy--Chenciner 1998}.  
We mention that it is not yet known to us what all central configurations are for $n=4$ 
for arbitrary choice of masses, and even the finiteness problem of their number has not been 
completely settled for $n=5$ (for generic choice of masses, this has been proven by Albouy 
and Kaloshin in \cite{Albouy--Kaloshin 2012}), and is yet largely open for $n \ge 6$.
Lecture notes from 2014 by Richard Moeckel on central configurations in a Euclidean space 
of arbitrary finite dimension give a nice overview with many more details (also on the history 
of the subject), and can be found on his website \cite{Moeckel 2014}. 

\section{Central configurations}

We now explain the concept of central configurations in $E^n$ and 
their associated homothetic, rigid and homographic motions of the Newtonian 
$n$-body problem in $E$.

\begin{definition}
For given masses, an $n$-body configuration $x\in E^n$ is called central with constant $\lambda$ if 
\[ \nabla U(x)=-\lambda xm \]
for some scalar $\lambda\in\mathbb{R}$. 
\end{definition}

Since $U(x)$ is homogeneous of degree $-1$,  we have 
\[ \tr(x^{\ast}\nabla U(x))=\mathcal{E}U(x)=-U(x) \] 
with $\mathcal{E}=\sum_{k} (x_{k}, \cdot) \nabla_{k} U(x)$ the Euler vector field on $E^n$, and therefore 
$$\lambda=U(x)/\tr(x^{\ast}xm)>0.$$ Note that central configurations are just the 
stationary points of the function $U(x)$ under the constraint 
$\tr(x^{\ast}xm)/2=1$. Clearly, if $x\in E^n$ is central with constant $\lambda$, 
then for all scalars $\mu>0$ and all proper rigidities $k\in\mathrm{SO}(E)$, 
the configuration $\mu kx\in E^n$ is again central with constant $\lambda/\mu^3$ .

\begin{proposition}
If $x\in E^n$ is a central configuration with constant $\lambda$ and $r(t)$ is a 
solution of the one dimensional Kepler problem $\ddot{r}=-\lambda/r^2$, then
$z(t)=r(t)x$ is a homothetic motion of the $n$-body problem. Conversely, any
homothetic solution $z(t)=r(t)x$ of the $n$-body problem can be expressed in this way for some 
central configuration $x\in E^n$.
\end{proposition}

\begin{proof} 
Indeed, if $x\in E^n$ is a central configuration with constant $\lambda$ 
and the real function $r(t)$ is a solution of $\ddot{r}=-\lambda/r^2$,  
then the motion $z(t)=r(t)x$ satisfies
\[ \ddot{z}m=\ddot{r}xm=-\lambda xm/r^2=\nabla U(x)/r^2=\nabla U(z), \]
since $\nabla U(x)$ is homogeneous in $x$ of degree $-2$.

Conversely, suppose $z(t)=r(t)x$ is
a solution of the $n$-body problem for some real function $r(t)$. By substitution into the equation of motion 
$\ddot{z}m=\nabla U(z)$, we obtain $r^2\ddot{r}xm=\nabla U(x)$. 
Hence $r^2\ddot{r}=-\lambda$ for some constant $\lambda\in\mathbb{R}$, 
and so $\nabla U(x)=-\lambda xm$, thus $x$ is a central configuration with constant $\lambda$. 
\end{proof}

We recall that a compatible complex structure on $E$ is a skew symmetric linear operator 
$J:E\rightarrow E$ with $J^2=-1$. A neccesary and sufficient condition for such
$J$ to exist is that $E$ has even dimension.

\begin{proposition}
Suppose $x\in E^n$ is a central configuration in $E$ with constant 
$\lambda=\omega^2>0$.
Any compatible complex structure $J$ on $E$ gives rise to a rigid motion 
$t\mapsto z(t)=\exp(t\omega J)x$ of the $n$-body problem.
Conversely, if $E$ is spanned by $x$, then any rigid motion
solution of $\ddot{z}m=\nabla U(z)$ is of this form.
\end{proposition}

\begin{proof}
Indeed, we have $\ddot{z}m=-\omega^2zm=\nabla U(z)$ since $z$ is central with 
scalar $\lambda=\omega^2$.

Conversely, if $Z:E\rightarrow E$ is a skew symmetric operator, then the
rigid motion $z(t)=\exp(tZ)x$ of the central configuration $x$ with scalar
$\lambda=\omega^2$ is a solution of $\ddot{z}m=\nabla U(z)$ if and only if 
$Z^2x=-\lambda x$. Since by assumption $E$ is spanned by $x$, we arrive
at $Z=\omega J$ with $J$ a compatible complex structure on $E$.
\end{proof}

Homothetic and rigid motions of a central configuration are both special cases
of the more general homographic motions.

\begin{theorem}
Suppose $t\mapsto(r,\theta)$ is a solution of the planar Kepler problem 
\[ \ddot{z}=-\lambda z/|z|^3\;,\;z=re^{i\theta}\;\Leftrightarrow\;
\ddot{r}-r\dot{\theta}^2=-\lambda/r^2\;,\;r\ddot{\theta}+2\dot{r}\dot{\theta}=0 \] 
in polar coordinates. If $x\in E^n$ is a central configuration with constant 
$\lambda$ and $J$ is a compatible complex structure on $E$, then  
\[ t\mapsto z(t)=r(t)\exp(\theta(t)J)x \]
is a homographic motion of the $n$-body problem.
\end{theorem}

\begin{proof}
This is just the standard derivation of the equations of motion for the Kepler problem
in polar coordinates. Indeed, let $x\in E^n$ be a fixed central configuration with 
constant $\lambda$, thus $\nabla U(x)=-\lambda xm$ holds by definition. 
We have to check that
\[ z=r(t) e^{\theta(t) J}x \]
is a solution of the equations of motion $\ddot{z}m=\nabla U(z)$ for the $n$-body problem. 
By differentiation, we have
\[ \ddot{z}=r^{-1}\{(\ddot{r}-r\dot{\theta}^2)+(r\ddot{\theta}+2\dot{r}\dot{\theta})J\}z, \]
and by assumption, we get $\ddot{z}=-\lambda r^{-3}z$. Since $z$ is central with
constant $\lambda / r^3$, we find $\nabla U(z)=-\lambda r^{-3}zm$, 
and hence $\ddot{z}m=\nabla U(z)$ is satisfied.
\end{proof}

Note that the term ``homographic'' in the terminology ``homographic motion'', though commonly used by celestial mechanists, should not be confused with the term ``homography'' in the geometric sense, which is synonymous to projective transformations. Under a homographic motion with negative total energy
\[ (\dot{r}+r\dot{\theta})^2/2-\lambda/r<0, \] 
each point particle $x_k\in E$ traverses a Kepler ellipse in the plane spanned 
by $\{x_k,Jx_k\}$ with one focus at the origin according to the area law, 
and all $n$ point particles traverse similar ellipses. 

We end this section by showing that central configurations exist in high
dimensions in abundance. Just take a (heavy) particle with mass $M$ at the origin 
$x_0=0$ and a cloud of $n$ (light) particles at positions $x_1,\cdots,x_n$ with 
equal masses $m$ with $\sum x_i=0$ and with a sufficient symmetry. 

\begin{theorem}
If $G$ be a finite irreducible subgroup of the orthogonal group $\mathrm{O}(E)$, 
such that $G$ acts transitively on the cloud $x_1,\cdots,x_n$ and for each
$i=1,\cdots,n$ the fixed point hyperplane in $E$ of the stabilizer group
$G_i$ of $x_i$ in $G$ is equal to the line $\mathbb{R}x_i$, then the
configuration $x=(x_0,x_1,\cdots,x_n)$ with masses $(M,m,\cdots,m)$ is
central.
\end{theorem} 

\begin{proof} 
The total force on the particle $x_i$ is the sum of the forces
expelled from the particles $x_j$ for $j\neq i$. Hence by symmetry
this total force on $x_i$ is fixed by $G_i$, and therefore equal to
$-\lambda_ix_i$ for some scalar $\lambda_i$. By symmetry we have
$\lambda_i=\lambda_j=\lambda$ for all $i,j\geq1,i\neq j$. We can take 
$\lambda_0=\lambda$ as well, and hence we find a central configuration.
\end{proof}

An example of such a configuration is obtained by taking for the
cloud the vertices of a regular polytope in the sense of Schl\"{a}fli
\cite{Coxeter 1973}. More generally, for any finite irreducible 
reflection group, one can take for the cloud the orbit of a nonzero
vector on an extremal ray of a positive Weyl chamber. For example,
in dimension $8$ one can obtain such a central configuration with a 
cloud of $483840$ particles with Weyl group symmetry of type $\mathrm{E}_8$.
But there are plenty of other examples, for example the minimal norm 4 vectors in 
the Leech lattice gives such a central configuration in dimension $24$
with a cloud of $196560$ particles (see \cite{Conway--Sloane 1988} for explanations of these lattices).

The planar central configurations with a regular $n$-gon for the cloud was deeply studied by Maxwell 
\cite{Maxwell 1859}, and more recently by Hall and Moeckel \cite{Hall 1993},
\cite{Moeckel 1994}. Their rigid motion is linearly stable for $n\geq7$ in case 
$m/M$ is sufficiently small (the larger $n$, the smaller $m/M$ should be). The 
question of linear stability of these general symmetric central configurations, in case 
$m/M$ is sufficiently small and for dimension at least $4$, is completely open. 
The motivation of Maxwell for this work was to understand the stability
of the rings of Saturn. His essay, published in 1859, was highly appreciated at the time, 
and won him the Adams prize for the year 1856.

\section{The spectra of the angular momenta}

Let $x\in E^n$ be a central configuration with constant $\lambda>0$. By a 
suitable positive scaling we may assume that $\lambda=1$, 
\emph{which will be assumed in this section.} For any compatible complex structure 
$J:E\rightarrow E$, we have discussed the rigid motion $t\mapsto z(t)=\exp(tJ)x$ 
of the Newtonian $n$-body problem. Note that $J^{\ast}=-J=J^{-1}$, so
$J$ is both skew symmetric and orthogonal. The conserved angular momentum 
\[ L:=\dot{z}mz^{\ast}-zm\dot{z}^{\ast}=Jxmx^{\ast}+xmx^{\ast}J \]
is a skew symmetric linear operator on $E$. The compatible complex structure
$J$ turns $E$ into a finite dimensional Hilbert space $(E,J)$ with Hermitian
form whose real part is the Euclidean inner product $(\cdot,\cdot)$. 
Clearly $L$ and $J$ commute, and if we write $X=xmx^{\ast}$ for the so called 
\emph{inertia operator} of the central configuration $x$, then
\[ K:=LJ^{\ast}=X+JXJ^{\ast} \] 
is a nonnegative selfadjoint operator on $(E,J)$. By definition, the \emph{real spectrum} 
of $L$ is the spectrum of $K$, considered as an ordered subset of $\mathbb{R}_+$ 
of cardinality equal to the complex dimension of the Hilbert space $(E,J)$.

What are the possible real spectra  of $L$ when $J$ varies over all the possible 
compatible complex structures on $E$? This question was posed by Chenciner, 
who conjectured it to be a convex polytope \cite{Chenciner 2010}, 
which was subsequently shown by an indirect argument by Chenciner and Jim\'enez-P\'erez 
\cite{Chenciner--Jiminez Perez 2011} by realizing this real spectum range between two Horn-type convex polytopes, and observe that a combinatorial lemma by Fomin--Fulton--Li--Poon \cite{Fomin--Fulton--Li--Poon 2005} affirms the coincidence of these two convex polytopes. 

The curious convexity property of this real spectrum range raised the question of finding for it a direct, conceptual proof, which is a question posed by Chenciner and Leclerc \cite{Chenciner--Leclerc 2014}. To present a direct proof of this convexity property, let us rephrase the question.

Let $j:E\rightarrow E$ be a fixed compatible complex structure.
Any compatible complex structure on $E$ is of the form $J=k^{\ast}jk$
for some $k\in\mathrm{O}(E)$, and therefore
\[ M:=kKk^{\ast}=(kXk^{\ast})+j(kXk^{\ast})j^{\ast} \]
is a nonnegative selfadjoint operator on the fixed Hilbert space $(E,j)$.
Let us write $\mathfrak{s}(E)$ for the space of symmetric operators
on $E$, and write $\mathfrak{s}(E,j)$ for its linear subspace
of selfadjoint operators on $(E,j)$. We consider $\mathfrak{s}(E)$
as Euclidean space with respect to the trace form $(Y,Z)=\tr(YZ)$,
and observe that $\mathrm{O}(E)$ acts on $\mathfrak{s}(E)$ by conjugation
as orthogonal linear transformations. Note that the map
\[ \mathfrak{s}(E)\rightarrow\mathfrak{s}(E,j),\;Y\mapsto(Y+jYj^{\ast})/2 \]
is nothing else but the orthogonal projection of $\mathfrak{s}(E)$ onto 
$\mathfrak{s}(E,j)$. Clearly this map is equivariant for the conjugation action 
of the unitary group $\mathrm{U}(E,j)$. Therefore the question on the range of the 
spectra of the selfadjoint 
operator $K$ on the Hilbert space
$(E,J)$ as $J$ varies over the space of all compatible complex structures on $E$ 
boils down to the determination of the image under the so called 
\emph{real moment map}
\[ \mu:\mathfrak{X}\rightarrow\mathfrak{s}(E,j),\;\mu(Y)=(Y+jYj^{\ast})/2 \] 
for the \emph{real Hamiltonian action} of the unitary group $\mathrm{U}(E,j)$ on 
the connected isospectral class $\mathfrak{X}=\{kXk^{\ast};k\in\mathrm{O}(E)\}$ 
in $\mathfrak{s}(E)$ of the inertia operator $X=xmx^{\ast}$ of the central 
configuration $x$. 

With these settings, the convexity result of Chenciner and Jim\'enez-P\'erez will be
an immediate consequence of a convexity theorem for the real moment polytope of O'Shea and Sjamaar 
\cite{O'Shea--Sjamaar 2000} for real reductive Lie algebras. Their result will be explained in the next 
section.

\section{The convexity theorem}

The real general linear Lie algebra $\mathfrak{gl}(E)$ of a Euclidean 
vector space $E$ has the standard Cartan involution 
$\theta:\mathfrak{gl}(E)\rightarrow\mathfrak{gl}(E)$
given by $\theta(X)=-X^{\ast}$, and the corresponding Cartan decomposition
\[ \mathfrak{gl}(E)=\mathfrak{so}(E)\oplus\mathfrak{s}(E) \]
as sum of $+1$ and $-1$ eigenspaces of $\theta$. The commutator bracket 
turns $\mathfrak{so}(E)$ in a Lie algebra, and $\mathfrak{s}(E)$ in a 
representation space for $\mathfrak{so}(E)$. 
The trace form $(X,Y)=\tr(XY)$ on $\mathfrak{gl}(E)$ is a nondegenerate
symmetric bilinear form, which is negative definite on $\mathfrak{so}(E)$ and 
positive definite on $\mathfrak{s}(E)$. The conjugation representation of 
$\mathrm{O}(E)$ on $\mathfrak{s}(E)$ is an orthogonal representation.

\begin{definition}
A real reductive Lie algebra with Cartan involution is a pair 
$(\mathfrak{g},\theta)$ with Lie subalgebra $\mathfrak{g}<\mathfrak{gl}(E)$ 
that is invariant under the standard Cartan involution $\theta$ of
$\mathfrak{gl}(E)$. By abuse of notation, the restriction of $\theta$ to 
$\mathfrak{g}$ is again denoted by $\theta$, and is called the Cartan 
involution of $\mathfrak{g}$. We have a corresponding Cartan decomposition
\[ \mathfrak{g}=\mathfrak{k}\oplus\mathfrak{s},\;
\mathfrak{k}=\mathfrak{g}\cap\mathfrak{so}(E),\;
\mathfrak{s}=\mathfrak{g}\cap\mathfrak{s}(E) \]
of $\mathfrak{g}$ as sum of $+1$ and $-1$ eigenspaces of $\theta$. 
The restriction of the trace form to $\mathfrak{g}$ is called the trace form
of $\mathfrak{g}$. It is a nondegenerate symmetric bilinear form, which is
negative on $\mathfrak{k}$ and positive on $\mathfrak{s}$. The connected 
Lie subgroup $K<\mathrm{SO}(E)$ with Lie algebra $\mathfrak{k}$ has a  
representation on $\mathfrak{s}$ by conjugation. Finally, we shall assume
that $K<\mathrm{SO}(E)$ is compact, so as to exclude the case of 
quasi-periodic subgroups. The connected Lie subgroup $G<\mathrm{GL}(E)$ 
with Lie algebra $\mathfrak{g}$ is a real reductive Lie group with $K$ 
as a maximal compact subgroup.
\end{definition}

\begin{example}
If $j:E\rightarrow E$ is a fixed complex structure on $E$ then the complex 
general linear Lie algebra $\mathfrak{gl}(E,j)$ gives, by restriction of scalars, 
a real reductive Lie algebra with Cartan involution.
\end{example}

\begin{definition}
A real reductive Lie algebra $(\mathfrak{g},\theta)$ is called complex if there is a 
complex structure $j:\mathfrak{g}\rightarrow\mathfrak{g}$ making $\mathfrak{g}$
into a complex Lie algebra, such that $j\theta=-\theta j$. This means that 
$\theta$ is an antilinear involution of $(\mathfrak{g},j)$. Note that
multiplication by $j$ interchanges $\mathfrak{k}$ and $\mathfrak{s}$.
\end{definition}

The complex general linear Lie algebra $(\mathfrak{gl}(E,j),\theta)$ is a 
natural example of a complex reductive Lie algebra with Cartan involution.

\begin{definition}
A real reductive Lie algebra $(\mathfrak{g},\theta)$ with Cartan decomposition
$\mathfrak{g}=\mathfrak{k}\oplus\mathfrak{s}$ has a natural complexification
$(\mathfrak{g}_c,\theta)$ defined by
\[ \mathfrak{g}_c=\mathfrak{g}+i\mathfrak{g},\;i=\sqrt{-1} \]
with Cartan decomposion
\[ \mathfrak{g}_c=\mathfrak{u}\oplus\mathfrak{p},\;
\mathfrak{u}=\mathfrak{k}+i\mathfrak{s},\;
\mathfrak{p}=\mathfrak{s}+i\mathfrak{k}\; \]
for the natural antilinear Cartan involution $\theta$ on $\mathfrak{g}_c$.
The homogeneous spaces $G/K$ and $U/K$ are dual (in the sense of \'{E}lie
Cartan) Riemannian symmetric spaces of noncompact and compact type respectively. 
Both spaces are different real forms of the complex symmetric space $G_c/K_c$ 
with transversal intersection at the base point $eK$. 
Here $G_c$ is the complex Lie subgroup of $\mathrm{GL} (E_{c})$ with Lie algebra $\mathfrak{g}_{c}$
(with $E_{c}$ the complexification of $E$), and
with $K_c$ the complex Lie subgroup of $G_c$   Lie algebra 
$\mathfrak{k}_c=\mathfrak{k}+i\mathfrak{k}$.
\end{definition}

The following theorem collects the standard structure theory for real reductive 
Lie algebras with Cartan involution \cite{Helgason 1978}.

\begin{theorem}
Let $(\mathfrak{g},\theta)$ be a real reductive Lie algebra with Cartan 
decomposition $\mathfrak{g}=\mathfrak{k}\oplus\mathfrak{s}$.
Any two maximal commutative linear subspaces in $\mathfrak{s}$ are conjugated under 
$K$. If $\mathfrak{a}<\mathfrak{s}$ is a fixed maximal commutative linear subspace,
then the Weyl group $W=N_K(\mathfrak{a})/C_K(\mathfrak{a})$ (normalizer modulo
centralizer of $\mathfrak{a}$ in $K$) acts by conjugation on $\mathfrak{a}$ 
as a finite reflection group. 
Let $\mathfrak{a}_+$ denote the closure of a fixed connected component of
the complement $\mathfrak{a}^{\circ}$ of all the mirrors in $\mathfrak{a}$,
and call it the (closed) positive Weyl chamber. Then $\mathfrak{a}_+$ is a strict
fundamental domain for the action of $W$ on $\mathfrak{a}$, and likewise for the conjugation
action of $K$ on $\mathfrak{s}$.
\end{theorem}

Let $\mathfrak{g}=\mathfrak{k}\oplus\mathfrak{s}$ be a real reductive Lie algebra
with complexification $\mathfrak{g}_c=\mathfrak{u}\oplus\mathfrak{p}$ as above.
For $X\in\mathfrak{a}_+$ we shall denote 
\[ \mathfrak{X}=\{kXk^{\ast};k\in K\}\subset\mathfrak{s} \]
and call it the isospectral class of $X$ in $\mathfrak{s}$. 
By construction, $\mathfrak{X}$ is connected, and $X=\mathfrak{X}\cap\mathfrak{a}_+$ 
by the above theorem. If we denote 
\[ \mathfrak{X}_c=\{uXu^{\ast};u\in U\}\subset\mathfrak{p} \]
then $\mathfrak{X}_c$ has the structure of a complex manifold with 
real form $\mathfrak{X}=\mathfrak{X}_c\cap\mathfrak{s}$.
Moreover $\mathfrak{X}_c$ has a K\"{a}hler metric, whose imaginary part
is the Kirillov-Kostant--Souriau symplectic form $\omega$ on $\mathfrak{X}_c$.
The action of $U$ on $\mathfrak{X}_c$ is Hamiltonian with moment map
the inclusion $\mathfrak{X}_c\hookrightarrow\mathfrak{p}$. For
this reason, we shall call the action of $K$ on the real form $\mathfrak{X}$
a real Hamiltonian action with real moment map the inclusion
$\mathfrak{X}\hookrightarrow\mathfrak{s}$.

We now have set up all the notations in order to formulate the convexity 
theorem of O'Shea and Sjamaar \cite{O'Shea--Sjamaar 2000} in case of a real reductive Lie algebra.

\begin{theorem} \label{Thm: real reductive O'Shea--Sjamaar}
Suppose $(\mathfrak{g},\theta)<(\mathfrak{g}',\theta')$ 
is a comparable pair of real reductive Lie algebras with Cartan involution. 
For $\mathfrak{X}'\subset\mathfrak{s}'$ a fixed isospectral class the orthogonal
projection $\mu:\mathfrak{X}'\rightarrow\mathfrak{s}$
is clearly equivariant for the conjugation action of $K$, and is called the real
moment map for the real Hamiltonian action of $K$ on $\mathfrak{X}'$. Under all
these assumptions, the intersection
\[ \mu(\mathfrak{X}')\cap\mathfrak{a}_+ \]
is a convex polytope, called the moment polytope of the real Hamiltonian action 
of $K$ on $\mathfrak{X}'$.
\end{theorem}

This theorem has a long history, and we shall mention just a few selected references. 
In case $(\mathfrak{g},\theta)<(\mathfrak{g}',\theta')$ are both complex reductive
Lie algebras with Cartan involution the theorem is due to Heckman \cite{Heckman 1982}.
The result was generalized by Guillemin and Sternberg, who replaced the coadjoint
orbit $\mathfrak{X}'$ of the overgroup $K'$ by a complex projective manifold with 
a Fubini--Study metric $h$ with a holomorphic linearizable action of $K$, which 
leaves the symplectic form $\omega=\Im{h}$ invariant, and $\mu$ the moment map
for this Hamiltonian action of $K$ \cite{Guillemin--Sternberg 1984}. This result was 
also obtained by Mumford, published in the appendix of a paper by 
Ness \cite{Ness 1984}.
This is the non-abelian convexity theorem in the K\"ahler case, which generalizes the former 
Abelian convexity theorem of Atiyah \cite{Atiyah 1982}, and that of Guillemin and Sternberg 
\cite{Guillemin--Sternberg 1982}. 
The proof of the general case without assuming the symplectic manifold to be K\"ahler was found 
by Kirwan \cite{Kirwan 1984}. These works were all done in the early eighties with 
many more exciting developments in moment map geometry.  

It took almost two decades before O'Shea and Sjamaar discovered the natural real setting 
of the convexity theorem, which generalizes the Abelian real convexity theorem of Duistermaat 
\cite{Duistermaat 1983}.

Indeed, consider the commutative diagram
\[ \begin{CD}
   \mathfrak{s}' \supset \mathfrak{X}' @>>> \mathfrak{X}'_c \subset \mathfrak{p}' \\
                 @VV{\mu}V                                  @V{\mu}VV             \\
   \mathfrak{s} \supset \mu(\mathfrak{X}') @>>> \mu(\mathfrak{X}'_c) \subset \mathfrak{p}
   \end{CD} \]
with $\mathfrak{X}'_c=\{uXu^{\ast};u\in U'\}$. As before, $\mathfrak{X}'_c$ can be
canonically identified with a coadjoint orbit of $U'$. Therefore it has a natural symplectic
form $\omega'$, for which the action of $U'$ by conjugation is Hamiltonian with moment map
the inclusion $\mathfrak{X}'_c\hookrightarrow\mathfrak{p}'$. The restriction of the action
from $U'$ to $U$ gives a moment map $\mu:\mathfrak{X}'_c\rightarrow\mathfrak{p}$, which 
is just the restriction of the orthogonal projection $\mathfrak{p}'\rightarrow\mathfrak{p}$.

The space $\mathfrak{X}'_c$ has a natural antisymplectic involution $\tau$, which is just the
restriction of the antiinvolution $-\theta'$ of $\mathfrak{p}'=\mathfrak{s}'\oplus i\mathfrak{k}'$,
taken $+1$ on $\mathfrak{s}'$ and $-1$ on $i\mathfrak{k}'$. In turn, the fixed point
locus of $\tau$ on $\mathfrak{X}'_c$ is just $\mathfrak{X}'=\mathfrak{X}'_c\cap\mathfrak{s}'$.
Hence the map $\mu:\mathfrak{X}'\rightarrow\mathfrak{s}$ is nothing but the restriction
of $\mu:\mathfrak{X}'_c\rightarrow\mathfrak{p}$ to the real form $\mathfrak{X}'$.
This explains our use of the terms \emph{real moment map} and \emph{real Hamiltonian action}.

If $\mathfrak{h}\subset\mathfrak{p}$ is a maximal commutative subspace with
$\mathfrak{h}\cap\mathfrak{s}=\mathfrak{a}$ and $\mathfrak{h}_+$ is an adapted
positive Weyl chamber, in the sense that $\mathfrak{h}_+\cap\mathfrak{a}=\mathfrak{a}_+$, then
\[ \mu(\mathfrak{X}_c')\cap\mathfrak{a}_+=
(\mu(\mathfrak{X}'_c)\cap\mathfrak{h}_+)\cap\mathfrak{a} \] 
is a convex polytope by the convexity theorem of Heckman.

Theorem \ref{Thm: real reductive O'Shea--Sjamaar} is therefore a direct consequence of the following result, 
which is also due to O'Shea and Sjamaar.

\begin{theorem} \label{Thm: Duistermaat}
We have $\mu(\mathfrak{X}')\cap\mathfrak{a}_+=\mu(\mathfrak{X}'_c)\cap\mathfrak{a}_+$\;.
\end{theorem}

We have restricted ourselves to the case of (co)adjoint orbits for a real reductive Lie 
algebra, which both suffices for our purpose and keeps the exposition as concrete as 
possible. In their paper, O'Shea and Sjamaar formulated everything in the general setting of a
Hamiltonian action of a connected compact Lie group $U$ on a connected symplectic manifold 
$(M,\omega)$. Suppose that the group $U$ has an
involution $\theta$ with fixed point 
group $K$, and the space $(M,\omega)$ has an antisymplectic involution $\tau$ with $M^{\tau}$ 
not empty. These two structures are assumed to be compatible, in the sense that 
$$\tau(ux)=\theta(u)\tau(x) \hbox{ and } \mu (\tau(x))=-\theta(\mu(x))$$ for all $u\in U$ and $x\in M$. Under these conditions, O'Shea and Sjamaar obtained the 
following general result

\begin{theorem}
We have $\mu(M^{\tau})\cap\mathfrak{a}_+=\mu(M)\cap\mathfrak{a}_+$ and the right hand side
\[ \mu(M)\cap\mathfrak{a}_+=(\mu(M)\cap\mathfrak{h}_+)\cap\mathfrak{a} \]
is indeed a convex polytope by the convexity theorem of Kirwan.
\end{theorem}

It is readily seen that this implies Theorem \ref{Thm: Duistermaat}.

\section{Pushforward of the normalized measure by the real moment map}
We start this section by explaining the notion of Gelfand pairs, and their associated harmonic analysis.

\subsection*{Harmonic analysis for Gelfand pairs}
\begin{definition}
A locally compact unimodular topological group $G$ with a compact subgroup $K<G$ is called
a \emph{Gelfand pair} if the natural unitary representation of $G$ on $L^2(G/K,dx)$ decomposes 
in a multiplicity free way. 
\end{definition}

It can be shown that this definition is equivalent to the following one:

\begin{definition}A pair $K<G$ is called a \emph{Gelfand pair}, if for any 
irreducible unitary representation $(V,\langle\cdot,\cdot\rangle)$ of $G$, the restriction 
from $G$ to $K$ contains the trivial representation of $K$ with multiplicity at most $1$. 
\end{definition}

\begin{definition}
For a Gelfand pair $K<G$, an irreducible unitary representation $(V,\langle\cdot,\cdot\rangle)$ of $G$ is called \emph{spherical} if $V^K=\mathbb{C}v$ has dimension $1$ for some $v\in V$ with $\langle v, v\rangle=1$. The function 
$$G \ni g\mapsto\phi_V(g)=\langle g v,v\rangle$$
is called the \emph{elementary spherical function} associated with the spherical 
representation $V$.
\end{definition}

Note that elementary spherical functions are normalized by 
$\phi_V(e)=1$. 

\begin{definition}Any function on $G$ that is both left and right invariant under $K$ 
 is called a \emph{spherical function}.
\end{definition}

Yet, a third equivalent definition for a Gelfand pair is the following:

\begin{definition}The pair $K<G$ is a \emph{Gelfand pair} if the 
Hecke algebra $\mathcal{H}(G/K)$ of continuous spherical functions on $G$ 
with compact support is commutative with respect to the convolution product. 
\end{definition}

The elementary spherical functions are the simultaneous eigenfunctions 
for the commutative algebra $\mathcal{H}(G/K)$, acting as convolution
integral operators on the space of spherical functions. 

Finally, in case that $G$ is a connected Lie group, there is a fourth equivalent definition for a Gelfand pair:

\begin{definition} For a connected Lie group $G$, the pair $K<G$ is a \emph{Gelfand pair} 
if and only if the algebra $\mathcal{D}(G/K)$ of linear differential operators on $G/K$, 
which are invariant under $G$, is commutative. 
\end{definition}

Similarly, the elementary spherical 
functions are the simultaneous eigenfunctions for the commutative algebra $\mathcal{D}(G/K)$, 
acting as invariant differential operators on the space of spherical functions.

Under all these equivalent conditions, the abstract spherical inversion theorem gives 
the existence of a unique positive measure $\mu_{\mathrm{P}}$ on the set $\widehat{G/K}$ of equivalence 
classes of unitary irreducible spherical representations of $K<G$, called the \emph{spherical 
Plancherel measure}, such that
\[ \phi(x)=\int_{\widehat{G/K}}\;\hat{\phi}(V)\phi_V(x)d\mu_{\mathrm{P}}(V) \]
for all $\phi\in\mathcal{H}(G/K)$, with
\[ \hat{\phi}(V)=\int_{G/K}\;\phi(x)\overline{\phi_V(x)}dx \]
the so called \emph{spherical Fourier transform} of $\phi\in\mathcal{H}(G/K)$. 

The case that $K$ 
is the trivial subgroup of $G=\mathbb{R}_+$ or $G=\{z\in\mathbb{C}^{\times};|z|=1\}$ gives 
the classical inversion formula for Fourier integrals and Fourier series respectively.

\subsection*{Harmonic analysis for Riemannian symmetric pairs }
After a brief exposition of the harmonic analysis for general Gelfand pairs, we now
come to certain particular cases of our interest. In the notation of the
previous section, these are the Riemannian symmetric space $G/K$ of noncompact type,
its compact dual Riemannian symmetric space $U/K$ and, finally, the intermediate flat 
tangent space $\mathfrak{s}$, considered as homogeneous space for the so called \emph{Cartan 
motion group} $\mathfrak{s}\rtimes K$. 

The spherical inversion formula was made explicit in the case $K<G$ by Harish-Chandra \cite{Harish-Chandra CP}
with simplifications by Helgason, Gangolli and Rosenberg (cf. \cite {Helgason 1984}).
Harish-Chandra enlarged the set $\widehat{G/K}$ of equivalence classes of spherical 
irreducible unitary representations of $G$ to the set $\widetilde{G/K}$ of equivalence 
classes of spherical continuous irreducible representations of $G$ on a Hilbert space, 
which are only unitary for the subgroup $K$. 
He showed that $\widetilde{G/K}\cong\mathfrak{a}_c/W$ and derived the 
\emph{Harish-Chandra isomorphism} $$\mathcal{D} (G/K) \cong S \mathfrak{a}_c^W, D\mapsto\gamma_{D} ,$$
in which $S\mathfrak{a}_c^W$ denotes the symmetric algebra of $\mathfrak{a}_c^W$. The associated elementary spherical
functions are given by the \emph{Harish-Chandra integral formula}
\[ \phi_{\lambda}(g)=\int_K\;a(gk)^{\lambda-\rho}dk=\int_K\;e^{(\lambda-\rho,A(gk))}dk \]
with Iwasawa decomposition $G=KAN,g=k(g)a(g)n(g)$,  Iwasawa projection $A(g)=\log{a(g)}$, 
the restricted Weyl vector $\rho$ (half sum of positive restricted roots counting multiplicities) and the normalized Haar measure $dk$ on $K$. 

These elementary spherical functions are solutions of the system of differential equations
\[ D\phi_{\lambda}=\gamma_{D}(\lambda)\phi_{\lambda},\quad D \in \mathcal{D} (G/K)\]
with normalization $\phi_{\lambda}(e)=1$ as before. The spherical inversion theorem now takes the form
\[ \phi(x)=\frac{1}{|W|}\int_{i\mathfrak{a}}\;
\hat{\phi}(\lambda)\phi_{\lambda}(x)\frac{d\mu_L(\lambda)}{|c(\lambda)|^2} \]
with spherical Fourier transform
\[ \hat{\phi}(\lambda)=\int_{G/K}\;\phi(x)\overline{\phi_{\lambda}(x)}dx, \]
the Lebesgue measure $\mu_L$ on $i\mathfrak{a}$ and the
Harish-Chandra c-function $\lambda\mapsto 
c(\lambda)$, given as an explicit product of $\Gamma$-factors by
the Gindikin--Karpelevic formula. 

The pair $K<\mathfrak{s}\rtimes K$, with the semidirect product $\mathfrak{s} \rtimes K$ acting on
$\mathfrak{s}$ via 
rotations and translations, is a Gelfand pair as well, and the group $\mathfrak{s}\rtimes K$
is called the Cartan motion group of the space $\mathfrak{s}$. 
The algebra $\mathcal{D}(\mathfrak{s})$ of invariant linear differential operators is isomorphic 
to 
the algebra $S\mathfrak{s}_c^K\cong S\mathfrak{a}_c^W$ of $K$-invariant linear 
differential operators on $\mathfrak{s}$ with constant coefficients.
Its simultaneous eigenfunctions are the symmetrized plane waves
\[ \psi_{\lambda}(X)=\int_K\;e^{(\lambda,kXk^{\ast})}dk \]
normalized by $\psi_{\lambda}(0)=1$ for all $\lambda\in\mathfrak{a}_c$ and $X\in\mathfrak{s}$.
The spherical inversion theorem is a direct consequence of the classical inversion theorem for the
Euclidean Fourier transform on $\mathfrak{s}$, applied for functions invariant under $K$. In a sense,
we can consider this theory on the flat space $\mathfrak{s}$ as the classical spectral limit
of the above Harish-Chandra theory for the curved space $G/K$, by the help of the following formula
\begin{proposition} We have
\[ \psi_{\lambda}(X)=\lim_{n\rightarrow\infty}\;\phi_{n\lambda}(\exp(X/n)) \]
for all $\lambda \in \mathfrak{a}_{c}$ and $X \in \mathfrak{s}$.
\end{proposition}
\begin{proof} By Harish-Chandra's integral formula
$$
\phi_{\lambda} (\exp X) = \int_{K} \;e^{(\lambda-\rho, A (\exp X \cdot k ))} d k
                                      = \int_{K} \;e^{(\lambda-\rho, A (\exp (Ad(k) X) ))} d k,                                       
$$
we have
$$\phi_{n \lambda} \bigl(\exp (X/n)\bigr) 
                                      = \int_{K} \;e^{(n \lambda-\rho, A (\exp (Ad(k) X/n) ))} d k.  $$
On $ \mathfrak{s}$, the infinitesimal Iwasawa projection  $\mathfrak{s} \to \mathfrak{a}$  coincides with the orthogonal    
projection  $\mathfrak{s} \to \mathfrak{a}$. Indeed, if $X \in \mathfrak{s}$ has infinitesimal Iwasawa decomposition $X=Y + H + Z$, for which 
$Y \in \mathfrak{k}, H \in \mathfrak{a}, Z \in \mathfrak{n}$, then $X=H+(Z-\theta Z)/2$, which means that  $H$ is also the orthogonal projection of $X$ 
on $\mathfrak{a}$, as we have the orthogonal decomposition $\mathfrak{s} = \mathfrak{a} \oplus (\mathfrak{s} \cap (\mathfrak{n} \oplus  \theta \mathfrak{n}))$.
We therefore deduce that
$$\lim_{n \rightarrow\infty} \phi_{n \lambda} \bigl(\exp (X/n)\bigr) 
                                      = \int_{K} \; e^{(\lambda, \lim_{n \rightarrow\infty} n A (exp (Ad(k) X/n)))} d k
                                      =  \int_{K} \;e^{( \lambda, Ad(k) X)} d k, $$                        
which by definition is equal to $\psi_{\lambda}(X)$, for all $\lambda \in \mathfrak{a}$, and thus for all $\lambda \in \mathfrak{a}_{c}$.
\end{proof}


The elementary spherical function $(\lambda,x)\mapsto \phi_{\lambda}(x)$ is holomorphic and 
Weyl group invariant in the spectral variable $\lambda\in\mathfrak{a}_c$, and real 
analytic in the space variable $x\in G/K$, or, in other words, holomorphic in the 
space variable $x$ taken from a suitable tubular neighborhood of $G/K$ in the 
complexified space $G_c/K_c$.
It has a holomorphic extension to all of $G_c/K_c$ if and only if $(\lambda-\rho)$
lies in the intersection $L \cap \mathfrak{a}_{+}$ of a suitable lattice $L$ with the positive chamber $\mathfrak{a}_{+ } \subset \mathfrak{a}$,  
given in explicit terms 
of the restricted root system by the Cartan--Helgason theorem (\cite[Ch.V, Theorem 4.1]{Helgason 1984}). The corresponding 
irreducible spherical representation $V(\lambda)$ (with highest weight $(\lambda- \rho)$) is then finite dimensional, and
unitary for the compact form $U$ of $G_c$. If $v\in V(\lambda)^K$ is a normalized
spherical vector, then $\phi_{\lambda}(u)=\langle uv,v\rangle$ for $u\in U$
with $\langle\cdot,\cdot\rangle$ the invariant Hermitian form on $V(\lambda)$.

\subsection*{Pushforward of the normalized measure}

After this survey of the theory of spherical functions,  we can finally explain the
meaning of the pushforward under the real moment map 
$\mu:\mathfrak{X}'\rightarrow\mathfrak{s}$ of the normalized invariant measure
on $\mathfrak{X}'$ in the notation of Theorem \ref{Thm: real reductive O'Shea--Sjamaar}
in terms of spherical representation theory. Let $\lambda'\in L'_+=(L'\cap\mathfrak{a}'_+) + \rho'$
and $(V(\lambda'),\langle\cdot,\cdot\rangle)$ be the associated finite dimensional
spherical irreducible unitary representation of $U'$ with normalized spherical vector 
$v'\in V(\lambda')^{K'}$. Let $\phi_{\lambda'}(u')=\langle u'v',v'\rangle$ be the 
associated elementary spherical function on $U'/K'$. Its restriction to the
totally geodesic submanifold $U/K<U'/K'$ is given by 
\[ \phi_{\lambda'}(u)=\sum_{\lambda\in L_+}\;m_{\lambda'}(\lambda)\phi_{\lambda}(u) \]
with $m_{\lambda'}(\lambda)=\langle v_{\lambda},v_{\lambda}\rangle$ if $v'=\sum_{\lambda}v_{\lambda}$
is the primary decomposition of $v'$ into components $v_{\lambda}$ for $\lambda\in L_+$ of 
spherical vectors for the Gelfand pair $K<U$. For $\lambda'\in L'_+$, let 
\[ \mu:\mathfrak{X}'_{\lambda'}=\{k'\lambda'k'^{\ast};k'\in K'\}\rightarrow \mathfrak{s} \]
be the real moment map for the real Hamiltonian action of $K$, and let $d\mathfrak{x}'$ be
the 
normalized $K'$-invariant measure on $\mathfrak{X}'_{\lambda'}$, so that $\int d\mathfrak{x}'=1$.

\begin{theorem}
Let $\mu\mapsto\delta(\mu-\lambda)$ be the Dirac delta distribution on $\mathfrak{a}$ with unit mass 
at $\lambda$. The probability measure $\mu_{\lambda'}$ on $\mathfrak{a}_+$  given by
\[ d\mu_{\lambda'}(\mu)=\lim_{n\rightarrow\infty}\;\sum_{\lambda\in L_+}\;m_{n\lambda'}(\lambda)\delta(\mu-\lambda/n) \]
describes the pushforward measure $\mu_{\ast}(d\mathfrak{x}')$ on $\mathfrak{s}$ by the relation
\[ \int_{\mathfrak{s}}\;f(\lambda)\mu_{\ast}(d\mathfrak{x}')(\lambda)=
\int_{\mathfrak{a}_+}\;f(\lambda)d\mu_{\lambda'}(\lambda) \]
for all continuous functions $f$ on $\mathfrak{s}$, which are invariant under $K$.
\end{theorem}

\begin{proof}
For $n\in\mathbb{N}$, $\lambda'\in L'_+$ and $X\in\mathfrak{s}$ we have
\begin{eqnarray*}
\phi_{n\lambda'}(\exp(X/n))=\sum_{\lambda\in L_+}\;m_{n\lambda'}(\lambda)\phi_{\lambda}(\exp(X/n)) \\
=\sum_{\lambda\in L_+/n}\;m_{n\lambda'}(n\lambda)\phi_{n\lambda}(\exp(X/n))\;,
\end{eqnarray*}
which in turn implies 
\[ \psi_{\lambda'}(X)=\int_{\mathfrak{a}_+}\;\psi_{\lambda}(X)d\mu_{\lambda'}(\lambda)  \]
for all $X\in\mathfrak{s}$. Hence the desired formula for $\mu_{\lambda'}$ as the 
pushforward measure $\mu_{\ast}(d\mathfrak{x}')$ follows from the Euclidean inversion
theorem for the flat space $\mathfrak{s}$ and the Fubini theorem.
\end{proof}

This theorem generalizes the result of \cite{Heckman 1982} on the relation between the
asymptotic behaviour of branching multiplicities and the pushforward of the Liouville
measure under the moment map in case $(\mathfrak{g},\theta)<(\mathfrak{g}',\theta')$
are both complex reductive Lie algebras with a Cartan involution. In that paper, the
convexity theorem was derived from the above theorem together with a 
simple representation-theoretic property. 

\subsection*{Some questions}

We end this section and the paper with some questions.
 
\begin{question}
For $\lambda\in L_+$ and $\lambda'\in L_+'$, does the spherical irreducible representation 
$V(\lambda)$ of $(\mathfrak{g},\theta)$ occur as subrepresentation of the 
spherical irreducible representation $V(\lambda')$ of $(\mathfrak{g}',\theta')$ if and 
only $m_{\lambda'}(\lambda)>0$? 
\end{question}

\begin{question}
Is it possible to generalize the 
results of this section to the general Hamiltonian setting, in line with O'Shea and Sjamaar?
\end{question}

\begin{question}
Is it possible to give an explicit description of the pushforward measure in the example of Chenciner?
\end{question}

\noindent
Gert Heckman, Radboud University Nijmegen: g.heckman@math.ru.nl
\newline
Lei Zhao, University of Groningen: l.zhao@rug.nl

\end{document}